\documentclass[12pt]{article}

\usepackage[T1]{fontenc}
\usepackage{lmodern}

\usepackage{amsmath,amssymb,amsthm,extarrows}
\usepackage{mathrsfs}
\allowdisplaybreaks[4]
\numberwithin{equation}{section}

\usepackage{graphicx}
\usepackage{xcolor}
\usepackage{float}

\usepackage{tikz}
\usetikzlibrary{
	arrows.meta,
	bending,
	calc,
	backgrounds,
	fit,
	positioning,
	decorations.pathmorphing,
	decorations.pathreplacing
}

\tikzset{
	redline/.style={
		red,
		thick,
		->,
		shorten >=1pt,
		shorten <=1pt
	}
}

\usepackage{array}
\usepackage{booktabs}
\usepackage{multirow}
\usepackage{makecell}
\usepackage{adjustbox}

\usepackage{enumitem}
\usepackage{indentfirst}

\usepackage{hyperref}
\hypersetup{colorlinks=true}

\theoremstyle{plain}
\newtheorem{theorem}{Theorem}[section]

\newtheorem{proposition}[theorem]{Proposition}
\newtheorem{conjecture}[theorem]{Conjecture}
\newtheorem{lemma}[theorem]{Lemma}

\definecolor{handred}{RGB}{220,0,0}
\definecolor{handgreen}{RGB}{0,150,0}

\renewcommand{\S}{\mathfrak{S}}
\newcommand{\TT}{\mathcal{T}}
\newcommand{\A}{\mathcal{A}}
\newcommand{\B}{\mathcal{B}}

\newcommand{\comp}{\mathsf{comp}}
\newcommand{\asc}{\mathsf{asc}}
\newcommand{\des}{\mathsf{des}}
\newcommand{\ldr}{\mathsf{ldr}}
\newcommand{\lmax}{\mathsf{lmax}}
\newcommand{\rmax}{\mathsf{rmax}}
\newcommand{\slmax}{\mathsf{slmax}}

\newcommand{\leaves}{\mathsf{leaves}}
\newcommand{\sub}{\mathsf{sub}}
\newcommand{\lsub}{\mathsf{lsub}}
\newcommand{\rpath}{\mathsf{rpath}}
\newcommand{\rootlab}{\mathsf{root}}
\newcommand{\lcap}{\mathsf{lcap}}

\newcommand{\std}{\mathsf{std}}
\newcommand{\emptyword}{\epsilon}

\begin{document}
	\begin{center}
		{\Large\bfseries
			On a conjecture concerning two-stack sortable permutations}
	\end{center}
	
	\begin{center}
		{\small  Yongzhou Wen,
			Sherry H.F. Yan$^{*}$}\\[4pt]
		
		Department of Mathematics, Zhejiang Normal University\\
		Jinhua 321004, P.R. China
	\end{center}
	
	\begingroup
	\renewcommand{\thefootnote}{\fnsymbol{footnote}}
	\footnotetext[1]{Corresponding author.\\
		\emph{E-mail addresses:}
		xxgz771@zjnu.edu.cn (Y. Wen),
		hfy@zjnu.cn (S.H.F. Yan).}
	\endgroup
	
	\vskip0.1in
	\noindent\textbf{Abstract.}
 The two-stack sortable permutations of length $n$ are known to be equinumerous with permutations of length $n$ avoiding the dashed patterns $3\text{-}1\text{-}4\text{-}2$ and $2\text{-}41\text{-}3$. In this paper, we confirm an intriguing conjecture posed by Claesson, Kitaev, and Steingr\'imsson, which asserts that the quadruple $(\comp,\asc,\ldr,\rmax)$ is equidistributed on these two permutation classes. Here $\comp$ and $\asc$ count the number of components and ascents, respectively, $\ldr$ denotes the length of the initial decreasing run, and $\rmax$ denotes the number of right-to-left maxima. This is accomplished  by constructing bijections between two-stack sortable permutations and $\beta(1,0)$-trees.
	\vskip0.1in
	\noindent\textbf{Keywords:}
	two-stack sortable permutation,  
	$\beta(1,0)$-tree, permutation statistic.

	\section{Introduction}
	
	The stack-sorting problem was introduced by Knuth~\cite{Knuth}
	and further studied by West~\cite{West}.
	Let $\S_n$ denote the set of permutations of
	$[n]:=\{1,2,\ldots,n\}$.
	The stack-sorting map $S$ can be defined recursively on words
	with distinct letters as follows~\cite{BousquetMelou}.
	Set $S(\emptyword)=\emptyword$, where $\emptyword$ denotes
	the empty word. Given a nonempty word $\pi$, write
	$\pi=\alpha m\beta$, where $m$ is the maximum letter of $\pi$
	and $\alpha,\beta$ are possibly empty words. Then
	\[
	S(\pi)=S(\alpha)S(\beta)m.
	\]
	A permutation $\pi\in\S_n$ is said to be
	\emph{two-stack sortable} if $S(S(\pi))=12\cdots n$.
	Denote by $\TT_n$ the set of all two-stack sortable permutations
	in $\S_n$.
	West~\cite{West} conjectured that
	\begin{equation}\label{eq:2ssp-enumeration}
		|\TT_n|=\frac{2(3n)!}{(n+1)!(2n+1)!},
		\qquad n\geq 1.
	\end{equation}
	This conjecture was first proved by Zeilberger~\cite{Zeilberger}
	using generating functions.
	Goulden and West~\cite{GouldenWest} later gave a bijective proof
	by establishing a correspondence between two-stack sortable
	permutations of length $n$ and rooted nonseparable planar maps
	with $n+1$ edges.

	Claesson, Kitaev, and Steingr\'imsson~\cite{CKS} studied
	permutations avoiding the two dashed patterns
	$3\text{-}1\text{-}4\text{-}2$ and $2\text{-}41\text{-}3$.
	For a word $w$ with distinct letters, let $\std(w)$ denote
	the permutation obtained from $w$ by replacing its $i$-th
	smallest letter by $i$.
	An occurrence of the pattern $3\text{-}1\text{-}4\text{-}2$
	in a permutation $\pi=\pi_1\pi_2\cdots\pi_n$ is a subsequence
	$\pi_i\pi_j\pi_k\pi_\ell$, where $i<j<k<\ell$, such that
	$\std(\pi_i\pi_j\pi_k\pi_\ell)=3142$.
	An occurrence of $2\text{-}41\text{-}3$ is a subsequence
	$\pi_i\pi_j\pi_{j+1}\pi_k$, where $i<j<j+1<k$, such that
	$\std(\pi_i\pi_j\pi_{j+1}\pi_k)=2413$.
	A permutation is said to avoid a pattern if it has no
	occurrences of it.
	Following~\cite{CKS}, we call a permutation avoiding both
	$3\text{-}1\text{-}4\text{-}2$ and $2\text{-}41\text{-}3$
	an \emph{avoider}. These permutations are the reverses of the nonseparable
	permutations studied by Dulucq, Gire, and West~\cite{DulucqGireWest}.
	Let $\A_n$ denote the set of all avoiders in $\S_n$. For example, $25314\in\A_5$, but $25134\notin\A_5$.

	Jacquard and Schaeffer~\cite{JacquardSchaeffer} established
	a bijection between rooted nonseparable planar maps and certain
	labeled plane trees, called $\beta(1,0)$-trees.
	Recall that a rooted plane tree is a rooted tree in which
	the children of each node are linearly ordered.
	A \emph{$\beta(1,0)$-tree} with at least one edge is
	a rooted plane tree whose nodes are labeled with positive
	integers satisfying  
	\begin{itemize}
	\item each leaf has label $1$, and the root has label equal
	to the sum of its children's labels;
	\item every other internal node has label no greater than
	the sum of its children's labels.
	\end{itemize}
	Denote by $\B_n$ the set of all $\beta(1,0)$-trees
	with $n$ edges.
	For $n=0$, set $\B_0=\{\circ\}$, where $\circ$
	denotes the tree consisting of a single  
  node labeled by $1$.
	For example, Fig.~\ref{fig:beta-tree-example} shows a tree
	in $\B_9$.
	
	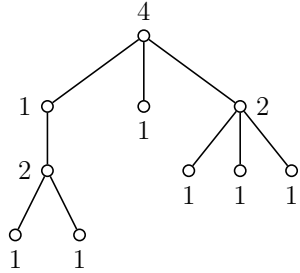
\begin{figure}[htbp]
		\centering
		\begin{tikzpicture}[
			scale=0.85,
			transform shape,
			x=1cm,y=1cm,
			vertex/.style={
				circle,draw,fill=white,
				inner sep=0pt,minimum size=5pt
			},
			every label/.style={font=\normalsize},
			line width=0.6pt
			]
			\node[vertex,label=above:$4$] (r) at (0,0) {};
			\node[vertex,label=left:$1$] (a) at (-1.5,-1.1) {};
			\node[vertex,label=below:$1$] (b) at (0,-1.1) {};
			\node[vertex,label=right:$2$] (c) at (1.5,-1.1) {};
			
			\node[vertex,label=left:$2$] (d) at (-1.5,-2.1) {};
			\node[vertex,label=below:$1$] (a1) at (-2,-3.1) {};
			\node[vertex,label=below:$1$] (a2) at (-1,-3.1) {};
			
			\node[vertex,label=below:$1$] (c1) at (0.7,-2.1) {};
			\node[vertex,label=below:$1$] (c2) at (1.5,-2.1) {};
			\node[vertex,label=below:$1$] (c3) at (2.3,-2.1) {};
			
			\draw (r)--(a) (r)--(b) (r)--(c)
			(a)--(d) (d)--(a1) (d)--(a2)
			(c)--(c1) (c)--(c2) (c)--(c3);
		\end{tikzpicture}
		\caption{A $\beta(1,0)$-tree with nine edges.}
		\label{fig:beta-tree-example}
	\end{figure}
	
	Following ~\cite{CKS},  for a $\beta(1,0)$-tree $B$, let $\sub(B)$ denote the number
	of children of the root, and let $\leaves(B)$ denote the
	number of leaves of $B$.
	The leftmost and rightmost paths of $B$ are the paths from
	the root to the leftmost and rightmost leaves, respectively.
	Let $\lsub(B)$ denote the number of nodes labeled $1$
	on the leftmost path, excluding the root, and let
	$\rpath(B)$ denote the number of edges on the rightmost
	path. 
	For $\circ$, we have
	\[
	\leaves(\circ)=1,
	\qquad
	\sub(\circ)=\lsub(\circ)=\rpath(\circ)=0.
	\]
	As an example, for   the tree $B$ in Fig.~\ref{fig:beta-tree-example}, we have
	\[
	(\sub,\leaves,\lsub,\rpath)B=(3,6,2,2).
	\]
	
	Given a permutation $\pi=\pi_1\pi_2\cdots\pi_n\in\S_n$,
	define the following permutation statistics:
	\[
	\begin{aligned}
		\asc(\pi)
		&=\bigl|\{i\mid \pi_i<\pi_{i+1},\ 1\leq i<n\}\bigr|,\\
		\ldr(\pi)
		&=\max\{i\mid \pi_1>\pi_2>\cdots>\pi_i,\ 1\leq i\leq n\},\\
		\rmax(\pi)
		&=\bigl|\{i\mid \pi_i>\pi_j
		\text{ for all }i<j\leq n,\ 1\leq i\leq n\}\bigr|.
	\end{aligned}
	\]
	For two words $\alpha$ and $\beta$, write $\alpha\prec\beta$
	if every letter of $\alpha$ is smaller than every letter
	of $\beta$.
	A \emph{component} of $\pi$ is a nonempty consecutive subword
	$\tau$ such that $\pi=\sigma\tau\rho$ with
	$\sigma\prec\tau\prec\rho$, and such that if
	$\tau=\alpha\beta$ with $\alpha\prec\beta$, then
	either $\alpha$ or $\beta$ is empty.
	Let $\comp(\pi)$ denote the number of components of $\pi$.
	For instance, the components of $213645$ are $21$, $3$,
	and $645$, and $\comp(213645)=3$. 
	
	In \cite{CKS}, Claesson, Kitaev, and Steingr\'imsson established a bijection between  $\B_n$ and $\A_n$  which verifies the following properties.

	\begin{proposition}[{\cite[Theorem~9]{CKS}}]
		\label{prop:CKS}
		For $n\geq1$, there exists a bijection
		$\phi:\B_n\longrightarrow\A_n$ such that for any
		$B\in\B_n$, we have
		\begin{equation}\label{eq:CKS-statistics}
			(\sub,\leaves,\lsub,\rpath)B
			=
			(\comp,1+\asc,\ldr,\rmax)\phi(B).
		\end{equation}
	\end{proposition}
	
	Dulucq, Gire, and Guibert~\cite{DulucqGireGuibert} proved that the pair $(\asc,\lmax)$ on avoiders is equidistributed  with the pair  $(\des,\rmax)$ on two-stack sortable permutations. This motivated Claesson, Kitaev, and Steingr\'imsson \cite{CKS} to pose the following intriguing conjecture.
	\begin{conjecture}[{\cite[Conjecture~13]{CKS}}]
		\label{conj:CKS}
		For $n\geq 1$, the quadruple
		$(\comp,\asc,\ldr,\rmax)$ has the same distribution
		on $\A_n$ and $\TT_n$.
	\end{conjecture}
	
	In this paper, we provide a bijective proof of
	Conjecture~\ref{conj:CKS} by constructing a    bijection between
	two-stack sortable permutations and $\beta(1,0)$-trees.
	
	\begin{theorem}\label{thm:main}
		For $n\geq 1$, there exists a bijection
		$F:\TT_n\longrightarrow\B_n$ such that for any $\pi\in\TT_n$,
		we have
		\begin{equation}\label{eq:main-statistics}
			(\comp,1+\asc,\ldr,\rmax)\pi
			=
			(\sub,\leaves,\lsub,\rpath)F(\pi).
		\end{equation}
	\end{theorem}
	By Proposition~\ref{prop:CKS} and Theorem~\ref{thm:main},
	the composition $\phi\circ F$ induces  a bijection between 
	$\TT_n$  and $\A_n$ that preserves the quadruple $(\comp,\asc,\ldr,\rmax)$, thereby confirming
 Conjecture~\ref{conj:CKS}.

	\section{The construction of the bijection $F$}
 In this section, we shall establish the bijection $F$ as stated in Theorem~\ref{thm:main}. To this end, we first recall the recursive decomposition of two-stack sortable permutations introduced by Fang~\cite{Fang} and the recursive decomposition of $\beta(1,0)$-trees introduced by Claesson, Kitaev, and Steingr\'imsson~\cite{CKS}. By employing the recursive decomposition of $\beta(1,0)$-trees, we construct a bijection $\psi\colon \TT_n\rightarrow \B_n$. Relying on the bijection $\psi$, we establish a bijection $\theta$ between indecomposable two-stack sortable permutations and indecomposable $\beta(1,0)$-trees which sends the triple $(1+\asc, \ldr, \rmax)$ on indecomposable two-stack sortable permutations to the triple $(\leaves, \lsub, \rpath)$ on the corresponding indecomposable $\beta(1,0)$-trees. Finally, by employing the direct-sum decomposition of two-stack sortable permutations, the recursive decomposition of $\beta(1,0)$-trees, and the bijection $\theta$, we establish the bijection $F\colon \TT_n\rightarrow \B_n$.
	
	\subsection{ A   decomposition of two-stack sortable permutations}
	\label{subsec:2ssp-decomposition}
	
	We first describe the recursive decomposition of
	two-stack sortable permutations introduced by
	Fang~\cite{Fang}.
	Let $\TT=\bigcup_{n\geq0}\TT_n$, with the convention that
	$\TT_0=\{\epsilon\}$, where $\epsilon$ denotes the empty
	permutation.
	We set $\asc(\epsilon)=\ldr(\epsilon)=\rmax(\epsilon)=0$.

	Given a permutation $\sigma=\sigma_1\sigma_2\ldots \sigma_n\in \mathfrak{S}_n$,  the entry $\sigma_i$ is said to be a {\em left-to-right maximum} if $\sigma_j<\sigma_i$ for all $j< i$. 
	For a permutation $\sigma$, let $\sigma^{+k}$ denote the sequence obtained by adding $k$ to every entry of $\sigma$. Similarly, let $\sigma^{+(k_1,m,k_2)}$ denote the sequence obtained from $\sigma$ by adding $k_1$ to each entry   smaller than $m$ and adding $k_2$ to all remaining entries.
	For example, if $\sigma=2417536$, then
	\[
	\sigma^{+2}=4\,6\,3\,9\,7\,5\,8,
	\qquad
	\sigma^{+(1,4,3)}=3\,7\,2\,(10) \,8\,4\,9.
	\]
	
	Let $\slmax(\sigma)$ denote the number of left-to-right maxima of $S(\sigma)$, with the convention that $\slmax(\emptyword)=0$.
	For example, if $\sigma=2417536$, then
	$S(\sigma)=2143567$, whose left-to-right maxima are
	$2,4,5,6,7$.
	Hence we have  $\slmax(\sigma)=5$.
	
	Given $\pi\in\TT_n$ with $n\geq1$, write
	$\pi=L\,n\,R$ and define
	\[
	\Delta(\pi)
	=
	\bigl(\std(L),\std(R)\bigr),
	\]
	where $\std(\emptyword)=\emptyword$.
	By~\cite[Proposition~3.2]{Fang}, both permutations
	in $\Delta(\pi)$ belong to $\TT$.
	
	Conversely, let $\alpha\in\TT_k$ and $\beta\in\TT_{\ell}$.
	Define
	\begin{equation}\label{eq:C1}
		C_1(\alpha,\beta)
		=
		\alpha\cdot(k+\ell+1)\cdot\beta^{+k},
	\end{equation}
	where $\alpha$ and $\beta$ may be empty.
	If both permutations are nonempty, let
	$a_1<a_2<\cdots<a_t$ be the left-to-right maxima of
	$S(\beta)$, where $t=\slmax(\beta)$.
	For $1\leq i\leq t$, define
	\begin{equation}\label{eq:C2}
		C_2(\alpha,\beta,i)
		=
		\alpha^{+(0,k,a_i)}
		\cdot(k+\ell+1)
		\cdot\beta^{+(k-1,a_i+1,k)}.
	\end{equation}
	For example, let $\alpha=231$ and $\beta=2417536$.
	Since $S(\beta)=2143567$, its left-to-right maxima are
	$a_1=2$, $a_2=4$, $a_3=5$, $a_4=6$ and $a_5=7$.
	Taking $i=2$, we have $a_i=4$ and
	\[
	\begin{aligned}
		C_1(231,2417536)
		&=231(11)574(10)869,\\
		C_2(231,2417536,2)
		&=271(11)463(10)859.
	\end{aligned}
	\]
	The following proposition follows
	from~\cite[Propositions~3.3--3.6]{Fang}.
	
	\begin{proposition}
		\label{prop:fang-decomposition}
		Given $\alpha\in\TT_k$ and $\beta\in\TT_{\ell}$
		with $k,\ell\geq0$, the permutation
		$C_1(\alpha,\beta)$ belongs to $\TT_{k+\ell+1}$ such that $\Delta(C_1(\alpha,\beta))=(\alpha,\beta)$,
		and 
		\begin{equation}\label{eq:C1-slmax}
			\slmax(C_1(\alpha,\beta))
			=\slmax(\alpha)+\slmax(\beta)+1.
		\end{equation}
		For $k,\ell\geq1$ and
		$1\leq i\leq\slmax(\beta)$, the permutation
		$C_2(\alpha,\beta,i)$ belongs to $\TT_{k+\ell+1}$,
		satisfying that  $\Delta(C_2(\alpha,\beta,i))=(\alpha,\beta)$,
		and 
		\begin{equation}\label{eq:C2-slmax}
			\slmax(C_2(\alpha,\beta,i))
			=\slmax(\alpha)+\slmax(\beta)-i+1.
		\end{equation}
		Moreover, every nonempty permutation in $\TT$ admits a unique decomposition of the form either $C_1(\alpha,\beta)$ or $C_2(\alpha,\beta,i)$.
	\end{proposition}

	For convenience, given $\alpha,\beta\in\TT$ and $0\leq i\leq \slmax(\beta)$, with $i=0$ whenever $\alpha=\epsilon$, write
	\[
	C(\alpha,\beta,i)=
	\begin{cases}
		C_1(\alpha,\beta),&\mbox{if $i=0$},\\
		C_2(\alpha,\beta,i),&\mbox{if $1\leq i\leq \slmax(\beta)$}.
	\end{cases}
	\]

	Let $\pi\in \TT_n$  and $\pi=C(\alpha, \beta, i)$ with $0\leq i\leq \slmax(\beta)$. 
	From~\eqref{eq:C1}--\eqref{eq:C2-slmax},  one can easily check that if $\alpha\neq \epsilon$, then 
	\begin{equation}\label{eq:C-statistics}
		\begin{aligned}
			\asc(\pi)&=\asc(\alpha)+\asc(\beta)+1,\\
			\ldr(\pi)&=\ldr(\alpha),\\
			\rmax(\pi)&=1+\rmax(\beta),\\
			\slmax(\pi)
			&=\slmax(\alpha)+\slmax(\beta)-i+1.
		\end{aligned}
	\end{equation}
	Moreover, if $\alpha=\epsilon$, 
  then
  \begin{equation}\label{eq:C-statistics-alpha}
  		\begin{aligned}
	\asc(\pi)&=\asc(\beta),\\
	\ldr(\pi)&=1+\ldr(\beta),\\
	\rmax(\pi)&=1+\rmax(\beta),\\
		\slmax(\pi)
	&=\slmax(\beta)+1.
		\end{aligned}
	\end{equation}
	
	We next describe the indecomposable permutations in $\TT$.
	Given permutations $\sigma$ and $\tau$, define their direct sum by
	\[
	\sigma\oplus\tau=\sigma\cdot\tau^{+|\sigma|}.
	\]
	We refer to $\sigma\oplus\tau$ as the \emph{direct sum} of $\sigma$ and $\tau$.
	A nonempty permutation is \emph{decomposable} if it admits a decomposition $\sigma\oplus\tau$ where both $\sigma$ and $\tau$ are nonempty; otherwise, it is \emph{indecomposable}.
	Since
	\[
	S(\sigma\oplus\tau)=S(\sigma)\oplus S(\tau),
	\]
	every nonempty $\pi\in\TT$ can be written uniquely as
	\[
	\pi=\pi^{(1)}\oplus\cdots\oplus\pi^{(r)},
	\]
	where each $\pi^{(j)}$ is an indecomposable permutation
	in $\TT$ and $r=\comp(\pi)$.
	Such a decomposition is called the {\em direct-sum decomposition} of $\pi$.
	
	The following property follows directly from the definition of the  indecomposable permutations. 
	\begin{lemma}
		\label{lem:indecomposable-2ssp}
		The permutation $C_1(\alpha,\beta)$ is indecomposable
		if and only if $\alpha=\epsilon$.
		The permutation $C_2(\alpha,\beta,i)$ is indecomposable
		if and only if $\alpha$ is indecomposable.
	\end{lemma}

	\subsection{A decomposition of $\beta(1,0)$-trees}
	\label{subsec:tree-decomposition}
	
	We next describe the decomposition of $\beta(1,0)$-trees
	introduced by Claesson, Kitaev, and Steingr\'imsson~\cite{CKS}.
	A $\beta(1,0)$-tree with at least one edge is called
	\emph{indecomposable} if its root has exactly one child. Otherwise, it is called {\em decomposable}.  
	Denote by $\overline{\B}_n$ the set of indecomposable
	trees in $\B_n$ for $n\geq1$.
	The tree $\circ$ is neither indecomposable nor decomposable.
	
	Let $U$ and $V$ be $\beta(1,0)$-trees, each with at least
	one edge. Define $U\oplus V$ to be the tree whose root
	label is the sum of the root labels of $U$ and $V$,
	and whose root subtrees are those of $U$ followed by
	those of $V$.
	
	For a $\beta(1,0)$-tree $U$, let $\rootlab(U)$ denote
	the label of its root.
	For an integer $i$ with $1\leq i\leq\rootlab(U)$,
	define $\lambda_i(U)$ to be the tree obtained from $U$ by joining a new root to the
	root of $U$ with an edge and assigning label $i$
	to both nodes.
	For example, Fig.~\ref{fig:lambda-tree} shows a tree
	$U$ with root label $3$ and the tree $\lambda_2(U)$.
	
	\begin{figure}[htbp]
		\centering
		\begin{tikzpicture}[
			scale=0.85,
			transform shape,
			x=1cm,y=1cm,
			vertex/.style={
				circle,draw,fill=white,
				inner sep=0pt,minimum size=5pt
			},
			every label/.style={font=\normalsize},
			line width=0.6pt
			]
			\node[vertex,label=above:$3$] (u) at (0,0) {};
			\node[vertex,label=below:$1$] (u1) at (-0.8,-0.9) {};
			\node[vertex,label=below:$1$] (u2) at (0,-0.9) {};
			\node[vertex,label=below:$1$] (u3) at (0.8,-0.9) {};
			\draw (u)--(u1) (u)--(u2) (u)--(u3);
			\node at (0,-1.65) {$U$};
			
			\draw[->] (1.4,-0.4)--(2.6,-0.4)
			node[midway,above] {$\lambda_2$};
			
			\node[vertex,label=above:$2$] (r) at (4,0.9) {};
			\node[vertex,label=right:$2$] (v) at (4,0) {};
			\node[vertex,label=below:$1$] (v1) at (3.2,-0.9) {};
			\node[vertex,label=below:$1$] (v2) at (4,-0.9) {};
			\node[vertex,label=below:$1$] (v3) at (4.8,-0.9) {};
			\draw (r)--(v)
			(v)--(v1) (v)--(v2) (v)--(v3);
			\node at (4,-1.65) {$\lambda_2(U)$};
		\end{tikzpicture}
		\caption{An example of the operation $\lambda_i$.}
		\label{fig:lambda-tree}
	\end{figure}
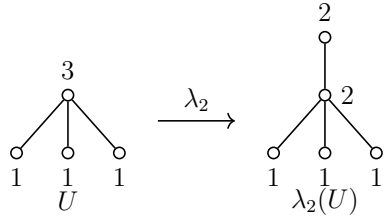
	
	The following proposition follows
	from~\cite[Section~3]{CKS}.
	
	\begin{proposition}
		\label{prop:cks-tree-decomposition}
		Let $B$ be a $\beta(1,0)$-tree with at least one edge.
		\begin{enumerate}
			\item
			The tree $B$ can be written uniquely as
			\[
			B=B_1\oplus\cdots\oplus B_r,
			\]
			where each $B_j$ is an indecomposable
			$\beta(1,0)$-tree and $r=\sub(B)$.
			
			\item
			If $B$ is indecomposable with at least two edges, then it can be
			written uniquely as
			\[
			B=\lambda_i(U),
			\qquad 1\leq i\leq\rootlab(U),
			\]
			where $U$ is a $\beta(1,0)$-tree and
			$i=\rootlab(B)$. If $B$ has only one edge, then $B=\lambda_1(\circ)$.
		\end{enumerate}
	\end{proposition}
	
	We will use the following properties of the tree operations
	(see~\cite[Lemmas~4 and~5]{CKS}).
	For $\beta(1,0)$-trees $U$ and $V$, each with at least
	one edge, and
	$1\leq i\leq\rootlab(U)$,  it is easily checked that
	\begin{equation}\label{eq:tree-sum-statistics}
		\begin{aligned}
			\sub(U\oplus V)&=\sub(U)+\sub(V),\\
			\leaves(U\oplus V)&=\leaves(U)+\leaves(V),\\
			\lsub(U\oplus V)&=\lsub(U),\\
			\rpath(U\oplus V)&=\rpath(V).
			\end{aligned}
			\end{equation}
				For a $\beta(1,0)$-tree $U$, we have
			\begin{equation}
				\label{eq:tree-sum-statistics-lambda}
					\begin{aligned}
			\sub(\lambda_i(U))&=1,\\
			\leaves(\lambda_i(U))&=\leaves(U),\\
			\lsub(\lambda_i(U))
			&=
			\begin{cases}
				1+\lsub(U),&i=1,\\
				\lsub(U),&i\geq2,
			\end{cases}\\
			\rpath(\lambda_i(U))&=1+\rpath(U).
		\end{aligned}
	\end{equation}

	\subsection{The bijection $\psi$}
	\label{subsec:psi}
	
	In this subsection, we construct a bijection
	$\psi:\TT_n\longrightarrow\B_n$ using the decomposition
	of two-stack sortable permutations presented in
	Subsection~\ref{subsec:2ssp-decomposition}.
	
	We first define a statistic on $\beta(1,0)$-trees.
	Given $B\in\B_n$ and an internal node $x$ on the leftmost path of $B$,  we define $\lcap_B(x)$ as the sum of the  labels of all children of $x$ other than  the leftmost child of $x$. Let the internal nodes on the leftmost path of $B$ be $r_1, r_2, \ldots, r_k$. We then set
	\[
	\lcap(B)=k+\sum_{i=1}^{k}\lcap_B(r_i).
	\]
	We call $\lcap(B)$ the \emph{left-path capacity} of $B$
	and set $\lcap(\circ)=0$. For example, for the tree $B$ in
	Fig.~\ref{fig:beta-tree-example}, we have
	$$\lcap(B)=3+(1+2+1)=7.$$
	
	We now define $\psi$ recursively.
	Set $\psi(\epsilon)=\circ$.
	
	By Proposition~\ref{prop:fang-decomposition},
	every nonempty permutation $\pi\in\TT$ can be written
	uniquely as $C(\alpha,\beta,i)$ with $0\leq i\leq \slmax(\beta)$.
	Let $L=\psi(\alpha)$ and $R=\psi(\beta)$.
	We consider the following two cases.
	
	\medskip
	\noindent
	\textbf{Case 1:} $\alpha=\epsilon$.\\
	In this case, $i=0$.
	Define $\psi(\pi)=\lambda_1(R)$.
	For example, let $\alpha=\epsilon$ and $\beta=12$.
	Then $\pi=C(\epsilon,12,0)=312$, and
	$\psi(312)=\lambda_1(\psi(12))$, as shown in
	Fig.~\ref{fig:psi-case1-example}.
	
	\begin{figure}[htbp]
		\centering
		\begin{tikzpicture}[
			scale=0.85,
			transform shape,
			x=1cm,y=1cm,
			vertex/.style={
				circle,draw,fill=white,
				inner sep=0pt,minimum size=5pt
			},
			every label/.style={font=\normalsize},
			tree caption/.style={
				align=center,anchor=north,inner sep=0pt
			},
			line width=0.6pt
			]
			\begin{scope}
				\node[vertex,label=right:$2$]
				(R0) at (0.00,-0.80) {};
				\node[vertex,label=right:$1$]
				(R1) at (-0.60,-1.60) {};
				\node[vertex,label=right:$1$]
				(R2) at (0.60,-1.60) {};
				\draw (R0)--(R1) (R0)--(R2);
				\node[tree caption] at (0.00,-2.05)
				{$R=\psi(12)$};
			\end{scope}
			
			\draw[->] (1.25,-0.80)--(3.15,-0.80)
			node[midway,above] {$\lambda_1$};
			
			\begin{scope}
				\node[vertex,label=right:$1$]
				(B0) at (4.80,0.00) {};
				\node[vertex,label=right:$1$]
				(B1) at (4.80,-0.80) {};
				\node[vertex,label=right:$1$]
				(B2) at (4.20,-1.60) {};
				\node[vertex,label=right:$1$]
				(B3) at (5.40,-1.60) {};
				\draw (B0)--(B1) (B1)--(B2) (B1)--(B3);
				\node[tree caption] at (4.80,-2.05)
				{$\psi(C(\epsilon,12,0))=\psi(312)$};
			\end{scope}
		\end{tikzpicture}
		\caption{The construction of
			$\psi(312)=\lambda_1(\psi(12))$ in Case~1.}
		\label{fig:psi-case1-example}
	\end{figure}
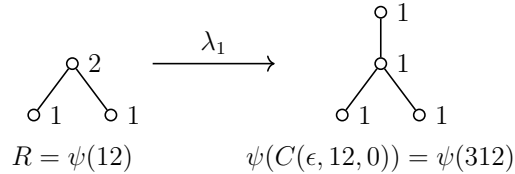

	\medskip
	\noindent
	\textbf{Case 2:} $\alpha\neq\epsilon$.\\
	Set $H=\lambda_1(R)$, and let
	$r_0,r_1,\ldots,r_{d+1}$ be the nodes on its leftmost
	path, listed from the root to the leaf.
	Here $r_0$ is the new root and $r_1$ is the original
	root of $R$.
	Let $x_s$ be the label of $r_s$, and set
	$a_s=\lcap_H(r_s)$ for $0\leq s\leq d$.
	Then $x_0=x_1=x_{d+1}=1$ and $a_0=0$.
	By the definition of $\lcap$, we have 
	\begin{equation}\label{eq:A-statistic}
		\lcap(H)=d+1+\sum_{s=1}^{d}a_s.
	\end{equation}
	Set
	\begin{equation}\label{eq:psi-blocks}
		b_s=1+x_{s+1}+a_s-x_s
		\qquad(1\leq s\leq d).
	\end{equation}
	These integers are positive and satisfy
	\begin{equation}\label{eq:psi-block-sum}
		\sum_{s=1}^{d}b_s
		=d+\sum_{s=1}^{d}a_s
		=\lcap(H)-1
		=\lcap(R).
	\end{equation}

	If $i=0$, set $(p,q)=(d+1,1)$.
	If $i\geq1$, let $(p,q)$ be the unique pair satisfying
	\begin{equation}\label{eq:psi-mark}
		i=\sum_{s=p}^{d}b_s-q+1,
		\qquad 1\leq p\leq d,\qquad 1\leq q\leq b_p.
	\end{equation}
	
	Let $l=\rootlab(L)$ and $j=l+q$.
	The tree $\psi(\pi)$ is obtained from $H$ as follows:
	\begin{enumerate}
		\item
		Attach the subtrees of the root of $L$ as the leftmost subtrees of $r_{p-1}$, preserving their relative order.
		
		\item
		Increase the label of $r_k$ by $j-1$ for all $0\leq k\leq p-1$, and increase the label of $r_p$ by $q-1$.
		All other labels remain unchanged.
	\end{enumerate}

	\medskip
	For example, let $L=\psi(12)$ and $R=\psi(15342)$.
	Since $\slmax(15342)=4$, we may take $0\leq i\leq4$
	in $C(12,15342,i)$.
	For $H=\lambda_1(R)$, we have $d=2$, $b_1=b_2=2$
	and $\lcap(H)=5$.
	The corresponding permutations and pairs $(p,q)$ are given by
	\[
	\begin{array}{c|c|c}
		i & C(12,15342,i) & (p,q)\\
		\hline
		0 & 12837564 & (3,1)\\
		1 & 13827564 & (2,2)\\
		2 & 15827463 & (2,1)\\
		3 & 16827453 & (1,2)\\
		4 & 17826453 & (1,1)
	\end{array}
	\]
	The five resulting trees are shown in
	Fig.~\ref{fig:psi-example}.
	
	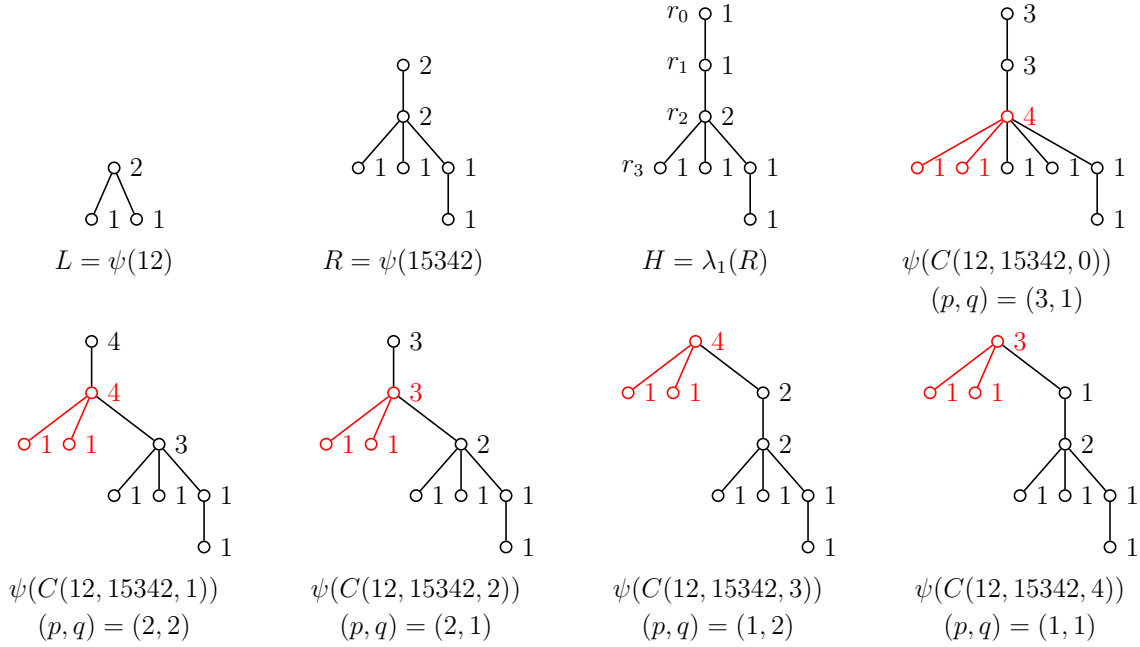
\begin{figure}[htbp]
		\centering
		\begin{tikzpicture}[
			scale=0.85,
			transform shape,
			x=1cm,y=1cm,
			vertex/.style={
				circle,draw,fill=white,
				inner sep=0pt,minimum size=5pt
			},
			every label/.style={font=\normalsize},
			tree caption/.style={
				align=center,anchor=north,inner sep=0pt
			},
			line width=0.6pt
			]
			\begin{scope}
				\node[vertex,label=right:$2$]
				(l0) at (0.00,-2.40) {};
				\node[vertex,label=right:$1$]
				(l1) at (-0.35,-3.20) {};
				\node[vertex,label=right:$1$]
				(l2) at (0.35,-3.20) {};
				\draw (l0)--(l1) (l0)--(l2);
				\node[tree caption] at (0.00,-3.65)
				{$L=\psi(12)$};
			\end{scope}
			
			\begin{scope}
				\node[vertex,label=right:$2$]
				(R0) at (4.50,-0.80) {};
				\node[vertex,label=right:$2$]
				(R1) at (4.50,-1.60) {};
				\node[vertex,label=right:$1$]
				(R2) at (3.80,-2.40) {};
				\node[vertex,label=right:$1$]
				(R3) at (4.50,-2.40) {};
				\node[vertex,label=right:$1$]
				(R4) at (5.20,-2.40) {};
				\node[vertex,label=right:$1$]
				(R5) at (5.20,-3.20) {};
				\draw (R0)--(R1)
				(R1)--(R2) (R1)--(R3) (R1)--(R4)
				(R4)--(R5);
				\node[tree caption] at (4.50,-3.65)
				{$R=\psi(15342)$};
			\end{scope}
			
			\begin{scope}
				\node[vertex,label=right:$1$,label=left:$r_0$]
				(h0) at (9.20,0.00) {};
				\node[vertex,label=right:$1$,label=left:$r_1$]
				(r0) at (9.20,-0.80) {};
				\node[vertex,label=right:$2$,label=left:$r_2$]
				(r1) at (9.20,-1.60) {};
				\node[vertex,label=right:$1$,label=left:$r_3$]
				(r2) at (8.50,-2.40) {};
				\node[vertex,label=right:$1$]
				(r3) at (9.20,-2.40) {};
				\node[vertex,label=right:$1$]
				(r4) at (9.90,-2.40) {};
				\node[vertex,label=right:$1$]
				(r5) at (9.90,-3.20) {};
				\draw (h0)--(r0) (r0)--(r1)
				(r1)--(r2) (r1)--(r3) (r1)--(r4)
				(r4)--(r5);
				\node[tree caption] at (9.20,-3.65)
				{$H=\lambda_1(R)$};
			\end{scope}
			
			\begin{scope}
				\node[vertex,label=right:$3$]
				(c0) at (13.90,0.00) {};
				\node[vertex,label=right:$3$]
				(c1) at (13.90,-0.80) {};
				\node[vertex,draw=red,label={[red]right:$4$}]
				(c2) at (13.90,-1.60) {};
				
				\node[vertex,draw=red,label={[red]right:$1$}]
				(c3) at (12.50,-2.40) {};
				
				\node[vertex,draw=red,label={[red]right:$1$}]
				(c4) at (13.20,-2.40) {};
				\node[vertex,label=right:$1$]
				(c5) at (13.90,-2.40) {};
				\node[vertex,label=right:$1$]
				(c6) at (14.60,-2.40) {};
				\node[vertex,label=right:$1$]
				(c7) at (15.30,-2.40) {};
				\node[vertex,label=right:$1$]
				(c8) at (15.30,-3.20) {};
				\draw (c0)--(c1) (c1)--(c2)
				 (c2)--(c5)
				(c2)--(c6) (c2)--(c7) (c7)--(c8);
				\draw[red]
				(c2)--(c3) (c2)--(c4);
				\node[tree caption] at (13.90,-3.65)
				{$\psi(C(12,15342,0))$\\[2pt]
					$(p,q)=(3,1)$};
			\end{scope}
			
			\begin{scope}
				\node[vertex,label=right:$4$]
				(t1n0) at (-0.35,-5.10) {};
				\node[vertex,draw=red,label={[red]right:$4$}]
				(t1n1) at (-0.35,-5.90) {};
				\node[vertex,draw=red,label={[red]right:$1$}]
				(t1n2) at (-1.40,-6.70) {};
				\node[vertex,draw=red,label={[red]right:$1$}]
				(t1n3) at (-0.70,-6.70) {};
				\node[vertex,label=right:$3$]
				(t1n4) at (0.70,-6.70) {};
				\node[vertex,label=right:$1$]
				(t1n5) at (0.00,-7.50) {};
				\node[vertex,label=right:$1$]
				(t1n6) at (0.70,-7.50) {};
				\node[vertex,label=right:$1$]
				(t1n7) at (1.40,-7.50) {};
				\node[vertex,label=right:$1$]
				(t1n8) at (1.40,-8.30) {};
				\draw (t1n0)--(t1n1)
				(t1n1)--(t1n4)
				(t1n4)--(t1n5) (t1n4)--(t1n6)
				(t1n4)--(t1n7) (t1n7)--(t1n8);
				\draw[red]
				(t1n1)--(t1n2) (t1n1)--(t1n3);
				\node[tree caption] at (0.00,-8.75)
				{$\psi(C(12,15342,1))$\\[2pt]
					$(p,q)=(2,2)$};
			\end{scope}
			
			\begin{scope}
				\node[vertex,label=right:$3$]
				(t2n0) at (4.35,-5.10) {};
				\node[vertex,draw=red,label={[red]right:$3$}]
				(t2n1) at (4.35,-5.90) {};
				\node[vertex,draw=red,label={[red]right:$1$}]
				(t2n2) at (3.30,-6.70) {};
				\node[vertex,draw=red,label={[red]right:$1$}]
				(t2n3) at (4.00,-6.70) {};
				\node[vertex,label=right:$2$]
				(t2n4) at (5.40,-6.70) {};
				\node[vertex,label=right:$1$]
				(t2n5) at (4.70,-7.50) {};
				\node[vertex,label=right:$1$]
				(t2n6) at (5.40,-7.50) {};
				\node[vertex,label=right:$1$]
				(t2n7) at (6.10,-7.50) {};
				\node[vertex,label=right:$1$]
				(t2n8) at (6.10,-8.30) {};
				\draw (t2n0)--(t2n1)
				(t2n1)--(t2n4)
				(t2n4)--(t2n5) (t2n4)--(t2n6)
				(t2n4)--(t2n7) (t2n7)--(t2n8);
				\draw[red]
				(t2n1)--(t2n2) (t2n1)--(t2n3);
				\node[tree caption] at (4.70,-8.75)
				{$\psi(C(12,15342,2))$\\[2pt]
					$(p,q)=(2,1)$};
			\end{scope}
			
			\begin{scope}
				\node[vertex,draw=red,label={[red]right:$4$}]
				(t3n0) at (9.05,-5.10) {};
				\node[vertex,draw=red,label={[red]right:$1$}]
				(t3n1) at (8.00,-5.90) {};
				\node[vertex,draw=red,label={[red]right:$1$}]
				(t3n2) at (8.70,-5.90) {};
				\node[vertex,label=right:$2$]
				(t3n3) at (10.10,-5.90) {};
				\node[vertex,label=right:$2$]
				(t3n4) at (10.10,-6.70) {};
				\node[vertex,label=right:$1$]
				(t3n5) at (9.40,-7.50) {};
				\node[vertex,label=right:$1$]
				(t3n6) at (10.10,-7.50) {};
				\node[vertex,label=right:$1$]
				(t3n7) at (10.80,-7.50) {};
				\node[vertex,label=right:$1$]
				(t3n8) at (10.80,-8.30) {};
				\draw 
				(t3n0)--(t3n3) (t3n3)--(t3n4)
				(t3n4)--(t3n5) (t3n4)--(t3n6)
				(t3n4)--(t3n7) (t3n7)--(t3n8);
				\draw[red]
				(t3n0)--(t3n1) (t3n0)--(t3n2);
				\node[tree caption] at (9.40,-8.75)
				{$\psi(C(12,15342,3))$\\[2pt]
					$(p,q)=(1,2)$};
			\end{scope}
			
			\begin{scope}
				\node[vertex,draw=red,label={[red]right:$3$}]
				(t4n0) at (13.75,-5.10) {};
				\node[vertex,draw=red,label={[red]right:$1$}]
				(t4n1) at (12.70,-5.90) {};
				\node[vertex,draw=red,label={[red]right:$1$}]
				(t4n2) at (13.40,-5.90) {};
				\node[vertex,label=right:$1$]
				(t4n3) at (14.80,-5.90) {};
				\node[vertex,label=right:$2$]
				(t4n4) at (14.80,-6.70) {};
				\node[vertex,label=right:$1$]
				(t4n5) at (14.10,-7.50) {};
				\node[vertex,label=right:$1$]
				(t4n6) at (14.80,-7.50) {};
				\node[vertex,label=right:$1$]
				(t4n7) at (15.50,-7.50) {};
				\node[vertex,label=right:$1$]
				(t4n8) at (15.50,-8.30) {};
				\draw 
				(t4n0)--(t4n3) (t4n3)--(t4n4)
				(t4n4)--(t4n5) (t4n4)--(t4n6)
				(t4n4)--(t4n7) (t4n7)--(t4n8);
				\draw[red]
				(t4n0)--(t4n1) (t4n0)--(t4n2);
				\node[tree caption] at (14.10,-8.75)
				{$\psi(C(12,15342,4))$\\[2pt]
					$(p,q)=(1,1)$};
			\end{scope}
		\end{tikzpicture}
		\caption{Examples of the construction of $\psi$
			with $L=\psi(12)$ and $R=\psi(15342)$.}
		\label{fig:psi-example}
	\end{figure}
	
	\begin{lemma}\label{prop:Gamma-statistics}
		For $n\geq1$ and $\pi\in\TT_n$,  we have  $\psi(\pi)\in\B_n$ and
		\begin{equation}\label{eq:Gamma-statistics}
			(\lcap,\rpath,\leaves,\lsub)\psi(\pi)
			=
			(\slmax,\rmax,1+\asc,\ldr)\pi.
		\end{equation}
	\end{lemma}
	
	\begin{proof}
		We proceed by induction on $n$.
		The assertion holds for $n=1$.
		Suppose that $n\geq2$ and the assertion holds
		for all smaller permutations.
		We use the notation in the definition of $\psi(\pi)$. We have two cases.
		
		\noindent{\bf Case 1:} $\alpha=\epsilon$.\\
		By the induction hypothesis, we have $R\in \B_{n-1}$ satisfying  
			\begin{equation}\label{eq:psi-induction-statistics-alpha}
		(\lcap,\rpath,\leaves,\lsub)R
	=
	(\slmax,\rmax,1+\asc,\ldr)\beta.
	\end{equation}
		Recall that $\psi(\pi)=\lambda_1(R)$.  Therefore, from \eqref{eq:C-statistics-alpha},  \eqref{eq:tree-sum-statistics-lambda} and \eqref{eq:psi-induction-statistics-alpha},    we have $\psi(\pi)\in \B_n$ satisfying 
		$$
		\begin{aligned}
			\lcap(\psi(\pi))=1+\lcap(R)=1+\slmax(\beta)=\slmax(\pi),\\
			\rpath(\psi(\pi))=1+\rpath(R)=1+\rmax(\beta)=\rmax(\pi),\\
			\leaves(\psi(\pi))=\leaves(R)=1+\asc(\beta)=1+\asc(\pi),\\
			\lsub(\psi(\pi))=1+\lsub(R)=1+\ldr(\beta)=\ldr(\pi).
			\end{aligned}
		$$
	This establishes~\eqref{eq:Gamma-statistics} in the case $\alpha=\epsilon$.

		\noindent{\bf Case 2:} $\alpha\neq \epsilon$.\\
		By the induction hypothesis,  we have
		$L\in\B_{|\alpha|}$ and $R\in\B_{|\beta|}$ satisfying 
		\begin{equation}\label{eq:psi-induction-statistics}
			\begin{aligned}
				\lcap(L)&=\slmax(\alpha),
				&\qquad \lcap(R)&=\slmax(\beta),\\
				\leaves(L)&=1+\asc(\alpha),
				&\qquad \leaves(R)&=1+\asc(\beta),\\
				\lsub(L)&=\ldr(\alpha),
				&\qquad \rpath(R)&=\rmax(\beta).
			\end{aligned}
		\end{equation}
		If $i=0$, then $(p,q)=(d+1,1)$.
		If $i\geq1$, then
		$1\leq i\leq\slmax(\beta)=\lcap(R)$.
		In this case, the positivity of the $b_s$ and
		\eqref{eq:psi-block-sum} ensure that there is a unique
		pair $(p,q)$ satisfying~\eqref{eq:psi-mark}.
		
		We first verify the labeling conditions.
		If $p\leq d$, then $q\leq b_p$ implies $x_p+q-1\leq x_{p+1}+a_p$.
		If $p=d+1$, then $q=1$, and the leaf $r_{d+1}$ retains label $1$.
		For each node among $r_0,\ldots,r_{p-1}$, both the node label and the sum of the labels of its children increase by $j-1$.
		Hence the resulting tree is a $\beta(1,0)$-tree with root label $j$ and $|\alpha|+|\beta|+1=n$ edges.
		
		Compared with $L$, the new leftmost path has
		$p-1$ additional edges, and the original node $r_p$
		becomes a non-leftmost child of $r_{p-1}$
		with label $x_p+q-1$.
		Hence, 
		\begin{equation}\label{eq:psi-lcap}
			\begin{aligned}
				\lcap(\psi(\pi))
				&=\lcap(L)+(p-1)
				+\sum_{s=1}^{p-1}a_s+x_p+q-1\\
				&=\lcap(L)+\sum_{s=1}^{p-1}b_s+q,
			\end{aligned}
		\end{equation}
		where the second equality follows from \eqref{eq:psi-blocks}.
		The inserted subtrees lie to the left of all
		original children of $r_{p-1}$, so the rightmost
		path of $H$ remains unchanged.
		All leaves of $L$ and $H$ are preserved.
		Moreover, the nodes preceding the inserted subtrees
		on the new leftmost path have labels at least
		$j\geq2$.
		Therefore,  we deduce that
		\begin{equation}\label{eq:psi-tree-statistics}
			\begin{aligned}
				\rpath(\psi(\pi))
				&=\rpath(H)=1+\rpath(R),\\
				\leaves(\psi(\pi))
				&=\leaves(L)+\leaves(H)
				=\leaves(L)+\leaves(R),\\
				\lsub(\psi(\pi))
				&=\lsub(L).
			\end{aligned}
		\end{equation}
		
		Notice that
		\begin{equation}\label{eq:psi-index-sum}
			\begin{aligned}
				\sum_{s=1}^{p-1}b_s+q
				&=\sum_{s=1}^{d}b_s-i+1\,\,\,\, &\mbox{(by  \eqref{eq:psi-mark})}\\
				&=\lcap(R)-i+1\,\,\,\, &\mbox{(by  \eqref{eq:psi-block-sum})}\\
				&
				=\slmax(\beta)-i+1. \,\,\,\, &\mbox{(by  	\eqref{eq:psi-induction-statistics})}
			\end{aligned}
		\end{equation}
		
		We now obtain
		\begin{equation}\label{eq:psi-permutation-statistics}
			\begin{array}{lll}
				\slmax(\pi)
				&=
				\slmax(\alpha)+\slmax(\beta)-i+1 &(\mbox{by \eqref{eq:C1-slmax},\,\eqref{eq:C2-slmax}})\\
				&=
				\slmax(\alpha)+\sum\limits_{s=1}^{p-1}b_s+q &(\mbox{by \eqref{eq:psi-index-sum}})\\
				&=\lcap(L)+\sum\limits_{s=1}^{p-1}b_s+q &(\mbox{by \eqref{eq:psi-induction-statistics}})\\
				&=
				\lcap(\psi(\pi))&(\mbox{by \eqref{eq:psi-lcap}}),\\[5pt]

				\rmax(\pi)
				&=
				1+\rmax(\beta)&(\mbox{by \eqref{eq:C-statistics}})\\
				& =1+\rpath(R) & (\mbox{by \eqref{eq:psi-induction-statistics}})\\
				&=\rpath(\psi(\pi))& (\mbox{by \eqref{eq:psi-tree-statistics}}),\\[5pt]

				1+\asc(\pi)
				&=
				2+\asc(\alpha)+\asc(\beta)& (\mbox{by \eqref{eq:C-statistics}})\\
				&=\leaves(L)+\leaves(R)& (\mbox{by \eqref{eq:psi-induction-statistics}})\\
				&=
				\leaves(\psi(\pi))&	(\mbox{by \eqref{eq:psi-tree-statistics}}),
			\end{array}
		\end{equation}
		and 
		$$
		\ldr(\pi)
		\xlongequal{\mbox{by \eqref{eq:C-statistics}}}
		\ldr(\alpha)
		\xlongequal{\mbox{by \eqref{eq:psi-induction-statistics}}}
		\lsub(L)
		\xlongequal{\mbox{by \eqref{eq:psi-tree-statistics}}}
		\lsub(\psi(\pi)).
		$$
		This proves~\eqref{eq:Gamma-statistics} in the case $\alpha\neq \epsilon$, completing the proof.
	\end{proof}

	\begin{proposition}\label{prop:psi}
		For $n\geq0$, the map
		$\psi:\TT_n\longrightarrow\B_n$ is a bijection.
	\end{proposition}
	
	\begin{proof}
		We first show by induction on $n$ that the map $\psi$ is injective; that is, we may recover a unique $\pi\in\TT_n$ from $\psi(\pi)=B$.
		For $n=0$, the tree $\circ$ corresponds to the permutation $\epsilon$.
		Let $B\in\B_n$ with $n\geq1$ and let $\pi=C(\alpha, \beta, i)$. In order to recover $\pi$, it suffices to recover $\alpha$, $\beta$ and $i$. 
		We distinguish two cases.
		
		\medskip
		\noindent
		\textbf{Case 1:} The root of $B$ has label $1$.\\
		The root has a unique child. Remove the root and replace the new root label by the sum of its children's labels.
		Denote the resulting tree by $R$. 
		If removing the root leaves a single node, we regard it as $\circ$. The same convention applies in Case 2.
		By the induction hypothesis, we can 
		recover $\beta$ from $R$  by letting $R=\psi(\beta)$, and set $\alpha=\epsilon$ and $i=0$.
		
		\medskip
		\noindent
		\textbf{Case 2:} The root of $B$ has label $j\geq2$.\\
		Starting from the root of $B$, find the topmost node $u$ on its leftmost path whose leftmost child has label less than $j$.
		Such a node exists, since the leftmost leaf has label $1$.
		The label of $u$ is at least $j$.
		
		Let $c_1,\ldots,c_t$ denote the labels of the children of $u$, ordered from left to right.
		Since $c_1<j\leq\sum\limits_{s=1}^{t}c_s$, there exists a unique integer $g\geq2$ such that
		\[
		\sum_{s=1}^{g-1}c_s<j\leq\sum_{s=1}^{g}c_s.
		\]
		Set $l=\sum_{s=1}^{g-1}c_s$ and $q=j-l$, and let $w$ be the $g$-th child of $u$.
		Then $1\leq q\leq c_g$.
		Recover $L$ and $H$ as follows:
		\begin{enumerate}
			\item Remove the leftmost $g-1$ subtrees of $u$ and attach their roots to a new root labeled $l$, preserving their relative order.
			Denote this tree by $L$.
			
			\item In the remaining tree, decrease the label of each node on the path from the root to $u$ by $j-1$, and decrease the label of $w$ by $q-1$.
			Denote the resulting tree by $H$.
		\end{enumerate}
		
		The labeling conditions for $L$ and $H$ follow from these operations.
		The root of $H$ and its unique child both have label $1$. 
		Remove the root of $H$ and replace the new root label by the sum of its children's labels to obtain $R$.
		Thus $H=\lambda_1(R)$, and both $L$ and $R$ are $\beta(1,0)$-trees with fewer edges than $B$.
		By the induction hypothesis, we may recover $\alpha$ and $\beta$ from $L$ and $R$  by letting $L=\psi(\alpha)$ and $R=\psi(\beta)$.
		
		Using the notation $r_s$ and $b_s$ for $H$ as in the construction of $\psi(\pi)$, one may verify that $u=r_{p-1}$ and $w=r_p$. If $w$ is a leaf, then $i=0$. Otherwise, we have
		\[
		i=\sum_{s=p}^{d}b_s-q+1,
		\]
		completing the proof that $\psi$ is injective.
		
		It remains to show that the map $\psi$ is surjective. This follows from the equality $|\TT_n|=|\B_n|$, which completes the proof.
	\end{proof}

	\subsection{The bijection $\theta$}
	\label{subsec:Omega}
	
	Let $\overline{\TT}_n$ denote the set of  indecomposable permutations  in $\TT_n$. 
	Relying on $\psi$, we describe a recursive map $\theta: \overline{\TT}_n\rightarrow \overline{\B}_n$.  For $n=1$,  define $\theta(1)$ to be the unique tree in $\overline{\B}_1$.     For $n\geq 2$, write  $\pi=C(\alpha, \beta, i)\in  \overline{\TT}_n$. By Lemma \ref{lem:indecomposable-2ssp}, we have $\alpha=\epsilon$ when $i=0$, and $\alpha$ is indecomposable and $1\leq i\leq \slmax(\beta)$  otherwise.

	\noindent{\bf Case 1:} $i=0$.\\
	Let $R=\psi(\beta)$	and set  $\theta(\pi)=\lambda_1(R)$.

	\noindent{\bf Case 2:} $i\neq 0$. \\
	In this case, we have  $1\leq i\leq \slmax(\beta)$.   Let $L=\theta(\alpha)$ and $R=\psi(\beta)$. 
	Set $H=\lambda_1(R)$.
	Let $r_0,r_1,\ldots,r_{d+1}$ be the nodes on the leftmost
	path of $H$, listed from the root to the leaf.
	Let $x_s$ be the label of $r_s$, and set
	$a_s=\lcap_H(r_s)$ for $0\leq s\leq d$.
	As in~\eqref{eq:psi-blocks}, set
	\[
	b_s=1+x_{s+1}+a_s-x_s
	\qquad(1\leq s\leq d).
	\]
	Let $(p,q)$ be the unique pair satisfying
	\begin{equation}\label{eq:Omega-mark}
		i=\sum_{s=p}^{d}b_s-q+1,
		\qquad 1\leq p\leq d,\qquad 1\leq q\leq b_p.
	\end{equation}
	Set $\ell=\rootlab(L)+q$.
	The tree $N=\theta(\pi)$ is obtained from $H$
	as follows:
	\begin{enumerate}
		\item
		Remove the root of $L$ and attach its unique subtree as 
		the leftmost subtree of $r_p$.
		
		\item
		Increase the label of each of
		$r_0,r_1,\ldots,r_p$ by $\ell-1$.
		All other labels remain unchanged.
	\end{enumerate}

	\medskip
	For example, let $\alpha=21$ and $\beta=13827564$.
	Set
	\[
	L=\theta(\alpha)=\lambda_1(\lambda_1(\circ)),
	\qquad
	R=\psi(\beta),
	\qquad
	H=\lambda_1(R).
	\]
	We have $\lcap(R)=6$, $d=2$, $b_1=4$ and $b_2=2$.
	The corresponding permutations and pairs $(p,q)$
	are given by
	\[
	\begin{array}{c|c|c}
		i & C(21,13827564,i) & (p,q)\\
		\hline
		1 & 31(11)25(10)49786 & (2,2)\\
		2 & 51(11)24(10)39786 & (2,1)\\
		3 & 71(11)24(10)39685 & (1,4)\\
		4 & 81(11)24(10)39675 & (1,3)\\
		5 & 91(11)24(10)38675 & (1,2)\\
		6 & (10)1(11)24938675 & (1,1)
	\end{array}
	\]
	Fig.~\ref{fig:Omega-example} shows $L$, $R$, $H$,
	$\theta(C(21,13827564,1))$ and
	$\theta(C(21,13827564,5))$.
	
	\begin{figure}[htbp]
		\centering
		\begin{tikzpicture}[
			scale=0.85,
			transform shape,
			x=1cm,y=1cm,
			vertex/.style={
				circle,draw,fill=white,
				inner sep=0pt,minimum size=5pt
			},
			every label/.style={font=\normalsize},
			tree caption/.style={
				align=center,anchor=north,inner sep=0pt
			},
			line width=0.6pt
			]
			
			% Input tree L
			\begin{scope}
				\node[vertex,label=right:$1$]
				(linput0) at (0,-2.25) {};
				\node[vertex,label=right:$1$]
				(linput1) at (0,-3.00) {};
				\node[vertex,label=right:$1$]
				(linput2) at (0,-3.75) {};
				
				\draw (linput0)--(linput1)--(linput2);
				
				\node[tree caption] at (0,-4.15)
				{$L=\theta(21)$};
			\end{scope}
			
			% Input tree R
			\begin{scope}[xshift=5cm]
				\node[vertex,label=right:$4$]
				(rinput0) at (-0.35,-0.75) {};
				\node[vertex,label=right:$4$]
				(rinput1) at (-0.35,-1.50) {};
				\node[vertex,label=right:$1$]
				(rinput2) at (-1.05,-2.25) {};
				\node[vertex,label=right:$1$]
				(rinput3) at (-0.35,-2.25) {};
				\node[vertex,label=right:$3$]
				(rinput4) at (1.05,-2.25) {};
				\node[vertex,label=right:$1$]
				(rinput5) at (0.35,-3.00) {};
				\node[vertex,label=right:$1$]
				(rinput6) at (1.05,-3.00) {};
				\node[vertex,label=right:$1$]
				(rinput7) at (1.75,-3.00) {};
				\node[vertex,label=right:$1$]
				(rinput8) at (1.75,-3.75) {};
				
				\draw
				(rinput0)--(rinput1)
				(rinput1)--(rinput2)
				(rinput1)--(rinput3)
				(rinput1)--(rinput4)
				(rinput4)--(rinput5)
				(rinput4)--(rinput6)
				(rinput4)--(rinput7)
				(rinput7)--(rinput8);
				
				\node[tree caption] at (0,-4.15)
				{$R=\psi(13827564)$};
			\end{scope}
			
			% Auxiliary tree H
			\begin{scope}[xshift=10cm]
				\node[vertex,label=right:$1$]
				(hinput0) at (-0.35,0.00) {};
				\node[vertex,label=right:$1$]
				(hinput1) at (-0.35,-0.75) {};
				\node[vertex,label=right:$4$]
				(hinput2) at (-0.35,-1.50) {};
				\node[vertex,label=right:$1$]
				(hinput3) at (-1.05,-2.25) {};
				\node[vertex,label=right:$1$]
				(hinput4) at (-0.35,-2.25) {};
				\node[vertex,label=right:$3$]
				(hinput5) at (1.05,-2.25) {};
				\node[vertex,label=right:$1$]
				(hinput6) at (0.35,-3.00) {};
				\node[vertex,label=right:$1$]
				(hinput7) at (1.05,-3.00) {};
				\node[vertex,label=right:$1$]
				(hinput8) at (1.75,-3.00) {};
				\node[vertex,label=right:$1$]
				(hinput9) at (1.75,-3.75) {};
				
				\draw
				(hinput0)--(hinput1)
				(hinput1)--(hinput2)
				(hinput2)--(hinput3)
				(hinput2)--(hinput4)
				(hinput2)--(hinput5)
				(hinput5)--(hinput6)
				(hinput5)--(hinput7)
				(hinput5)--(hinput8)
				(hinput8)--(hinput9);
				
				\node[left=3pt] at (hinput0) {$r_0$};
				\node[left=3pt] at (hinput1) {$r_1$};
				\node[left=3pt] at (hinput2) {$r_2$};
				\node[left=3pt] at (hinput3) {$r_3$};
				
				\node[tree caption] at (0,-4.15)
				{$H=\lambda_1(R)$};
			\end{scope}
			
			% Image for i=1
			\begin{scope}[xshift=2.5cm,yshift=-5.25cm]
				\node[vertex,draw=black,label={[text=black]right:$3$}]
				(n0) at (-0.35,0.00) {};
				\node[vertex,draw=black,label={[text=black]right:$3$}]
				(n1) at (-0.35,-0.75) {};
				\node[vertex,draw=red,label={[text=red]right:$6$}]
				(n2) at (-0.35,-1.50) {};
				\node[vertex,draw=red,label={[text=red]right:$1$}]
				(n3) at (-1.75,-2.25) {};
				\node[vertex,draw=red,label={[text=red]right:$1$}]
				(n4) at (-1.75,-3.00) {};
				\node[vertex,draw=black,label={[text=black]right:$1$}]
				(n5) at (-1.05,-2.25) {};
				\node[vertex,draw=black,label={[text=black]right:$1$}]
				(n6) at (-0.35,-2.25) {};
				\node[vertex,draw=black,label={[text=black]right:$3$}]
				(n7) at (1.05,-2.25) {};
				\node[vertex,draw=black,label={[text=black]right:$1$}]
				(n8) at (0.35,-3.00) {};
				\node[vertex,draw=black,label={[text=black]right:$1$}]
				(n9) at (1.05,-3.00) {};
				\node[vertex,draw=black,label={[text=black]right:$1$}]
				(n10) at (1.75,-3.00) {};
				\node[vertex,draw=black,label={[text=black]right:$1$}]
				(n11) at (1.75,-3.75) {};
				
				\draw[black] (n0)--(n1);
				\draw[black] (n1)--(n2);
				\draw[red] (n2)--(n3);
				\draw[red] (n3)--(n4);
				\draw[black] (n2)--(n5);
				\draw[black] (n2)--(n6);
				\draw[black] (n2)--(n7);
				\draw[black] (n7)--(n8);
				\draw[black] (n7)--(n9);
				\draw[black] (n7)--(n10);
				\draw[black] (n10)--(n11);
				
				\node[tree caption] at (0,-4.15)
				{$\theta(C(21,13827564,1))$\\[2pt]
					$(p,q)=(2,2)$};
			\end{scope}
			
			% Image for i=5
			\begin{scope}[xshift=7.5cm,yshift=-5.25cm]
				\node[vertex,draw=black,label={[text=black]right:$3$}]
				(v0) at (-0.88,0.00) {};
				\node[vertex,draw=red,label={[text=red]right:$3$}]
				(v1) at (-0.88,-0.75) {};
				\node[vertex,draw=red,label={[text=red]right:$1$}]
				(v2) at (-1.75,-1.50) {};
				\node[vertex,draw=red,label={[text=red]right:$1$}]
				(v3) at (-1.75,-2.25) {};
				\node[vertex,draw=black,label={[text=black]right:$4$}]
				(v4) at (0.00,-1.50) {};
				\node[vertex,draw=black,label={[text=black]right:$1$}]
				(v5) at (-1.05,-2.25) {};
				\node[vertex,draw=black,label={[text=black]right:$1$}]
				(v6) at (-0.35,-2.25) {};
				\node[vertex,draw=black,label={[text=black]right:$3$}]
				(v7) at (1.05,-2.25) {};
				\node[vertex,draw=black,label={[text=black]right:$1$}]
				(v8) at (0.35,-3.00) {};
				\node[vertex,draw=black,label={[text=black]right:$1$}]
				(v9) at (1.05,-3.00) {};
				\node[vertex,draw=black,label={[text=black]right:$1$}]
				(v10) at (1.75,-3.00) {};
				\node[vertex,draw=black,label={[text=black]right:$1$}]
				(v11) at (1.75,-3.75) {};
				
				\draw[black] (v0)--(v1);
				\draw[red] (v1)--(v2);
				\draw[red] (v2)--(v3);
				\draw[black] (v1)--(v4);
				\draw[black] (v4)--(v5);
				\draw[black] (v4)--(v6);
				\draw[black] (v4)--(v7);
				\draw[black] (v7)--(v8);
				\draw[black] (v7)--(v9);
				\draw[black] (v7)--(v10);
				\draw[black] (v10)--(v11);
				
				\node[tree caption] at (0,-4.15)
				{$\theta(C(21,13827564,5))$\\[2pt]
					$(p,q)=(1,2)$};
			\end{scope}
			
		\end{tikzpicture}
		\caption{The construction of $\theta(C(\alpha,\beta,i))$
			for $\alpha=21$, $\beta=13827564$ and $i=1,5$.}
		\label{fig:Omega-example}
	\end{figure}
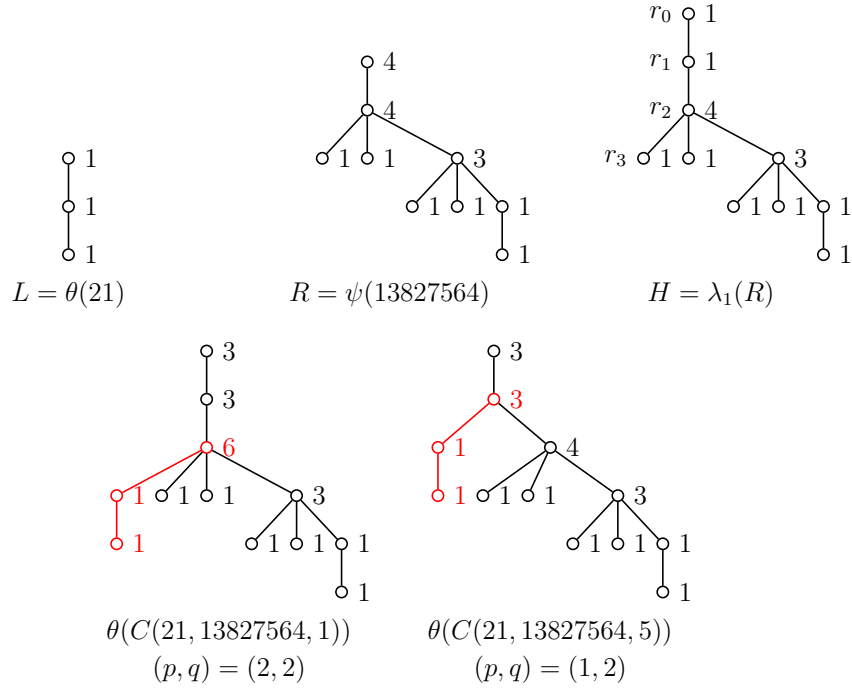
	\medskip
	
	\begin{lemma}\label{lem:theta}
		For $n\geq1$ and $\pi\in\overline{\TT}_n$, we have
		$\theta(\pi)\in\overline{\B}_n$ and
		\begin{equation}\label{eq:Omega-statistics}
			(\rpath,\leaves,\lsub)\theta(\pi)
			=
			(\rmax,1+\asc,\ldr)\pi.
		\end{equation}
	\end{lemma}
	
	\begin{proof}
		We proceed by induction on $n$.
		The assertion holds for $n=1$.
		Suppose that $n\geq2$ and the assertion holds
		for all smaller indecomposable permutations.
		We use the notation in the definition of $\theta(\pi)$.
		We have two cases.
		
		\noindent{\bf Case 1:} $\alpha=\epsilon$.\\
		By  Lemma~\ref{prop:Gamma-statistics}, we have $R\in \B_{n-1}$ satisfying  
		\begin{equation}\label{eq:psi-induction-statistics-alpha-theta}
			(\rpath,\leaves,\lsub)R
			=
			(\rmax,1+\asc,\ldr)\beta.
		\end{equation}
		Recall that $\theta(\pi)=\lambda_1(R)$.  Therefore, from \eqref{eq:C-statistics-alpha},  \eqref{eq:tree-sum-statistics-lambda} and \eqref{eq:psi-induction-statistics-alpha-theta},    we have $\theta(\pi)\in \overline{\B}_n$ satisfying 
		$$
		\begin{aligned}	
			\rpath(\theta(\pi))=1+\rpath(R)=1+\rmax(\beta)=\rmax(\pi),\\
			\leaves(\theta(\pi))=\leaves(R)=1+\asc(\beta)=1+\asc(\pi),\\
			\lsub(\theta(\pi))=1+\lsub(R)=1+\ldr(\beta)=\ldr(\pi).
		\end{aligned}
		$$
		This proves~\eqref{eq:Omega-statistics} in the case $\alpha=\epsilon$.

		\noindent{\bf Case 2:} $\alpha\neq \epsilon$.\\
		By the induction hypothesis and Lemma  \ref{prop:Gamma-statistics}, 
		we have $L\in\overline{\B}_{|\alpha|}$ and
		$R\in\B_{|\beta|}$ satisfying 
		\begin{equation}\label{eq:Omega-psi-statistics}
			\begin{aligned}
				\lcap(R)&=\slmax(\beta),\\
				\leaves(L)&=1+\asc(\alpha),
				&\qquad \leaves(R)&=1+\asc(\beta),\\
				\lsub(L)&=\ldr(\alpha),
				&\qquad \rpath(R)&=\rmax(\beta).
			\end{aligned}
		\end{equation}
		We have $1\leq i\leq\slmax(\beta)=\lcap(R)$.
		The positivity of the $b_s$ and~\eqref{eq:psi-block-sum}
		ensure that there is a unique pair $(p,q)$
		satisfying~\eqref{eq:Omega-mark}.
		
		We first verify the labeling conditions.
		The unique child of the root of $L$ has label
		$\rootlab(L)$.
		At node $r_p$,    $q\leq b_p$ implies
		\[
		x_p+\ell-1
		=x_p+\rootlab(L)+q-1
		\leq x_{p+1}+a_p+\rootlab(L).
		\]
		For each node among $r_0,\ldots,r_{p-1}$,
		both the node label and the sum of the labels
		of its children increase by $\ell-1$.
		Since $p\geq1$, the root retains its unique child.
		Hence the resulting tree is an indecomposable
		$\beta(1,0)$-tree with root label $\ell$
		and $|\alpha|+|\beta|+1=n$ edges.
		
		The inserted subtree lies to the left of all
		original children of $r_p$, so the rightmost
		path of $H$ remains unchanged.
		All leaves of $L$ and $H$ are preserved.
		Moreover, the nodes preceding the inserted subtree
		on the new leftmost path have labels at least
		$\ell\geq2$.
		Therefore,
		\begin{equation}\label{eq:Omega-tree-statistics}
			\begin{aligned}
				\rpath(N)
				&=\rpath(H)=1+\rpath(R),\\
				\leaves(N)
				&=\leaves(L)+\leaves(H)
				=\leaves(L)+\leaves(R),\\
				\lsub(N)&=\lsub(L).
			\end{aligned}
		\end{equation}
		
		We now obtain
		\begin{equation}\label{eq:Omega-permutation-statistics}
			\begin{array}{lll}
				\rmax(\pi)
				&=1+\rmax(\beta)
				&\mbox{(by \eqref{eq:C-statistics})}\\
				&=1+\rpath(R)
				&\mbox{(by \eqref{eq:Omega-psi-statistics})}\\
				&=\rpath(N)
				&\mbox{(by \eqref{eq:Omega-tree-statistics})},\\[5pt]
				1+\asc(\pi)
				&=2+\asc(\alpha)+\asc(\beta)
				&\mbox{(by \eqref{eq:C-statistics})}\\
				&=\leaves(L)+\leaves(R)
				&\mbox{(by \eqref{eq:Omega-psi-statistics})}\\
				&=\leaves(N)
				&\mbox{(by \eqref{eq:Omega-tree-statistics})},
			\end{array}
		\end{equation}
		and
		\[
		\ldr(\pi)
		\xlongequal{\mbox{by \eqref{eq:C-statistics}}}
		\ldr(\alpha)
		\xlongequal{\mbox{by \eqref{eq:Omega-psi-statistics}}}
		\lsub(L)
		\xlongequal{\mbox{by \eqref{eq:Omega-tree-statistics}}}
		\lsub(N).
		\]
		This establishes ~\eqref{eq:Omega-statistics} in the case $\alpha\neq \epsilon$, completing the proof.
	\end{proof}

		\begin{proposition}\label{prop:theta}
			For $n\geq1$, the map
			$\theta:\overline{\TT}_n\longrightarrow\overline{\B}_n$
			is a bijection.
		\end{proposition}
		
		\begin{proof}
			We first show by induction on $n$ that the map $\theta$
			is injective; that is, we may recover a unique
			$\pi\in\overline{\TT}_n$ from $\theta(\pi)=N$.
			For $n=1$, the unique tree corresponds to the permutation $1$.
			Let $\pi=C(\alpha,\beta,i)\in\overline{\TT}_n$, where
			$n\geq2$, and set $N=\theta(\pi)$.
			In order to recover $\pi$, it suffices to recover
			$\alpha$, $\beta$ and $i$.
			We distinguish two cases.
			
			\medskip
			\noindent
			\textbf{Case 1:} The root of $N$ has label $1$.
			
			The root has a unique child.
			Remove the root and replace the new root label
			by the sum of its children's labels.
			Denote the resulting tree by $R$.
			By Proposition~\ref{prop:psi}, we can  recover $\beta$ from $R$ by letting $R=\psi(\beta)$,
			and set $\alpha=\epsilon$ and $i=0$.
			
			\medskip
			\noindent
			\textbf{Case 2:} The root of $N$ has label $\ell\geq2$.
			
			Starting from the root of $N$, find the topmost node $z$
			on its leftmost path whose label is less than $\ell$.
			Such a node exists, since the leftmost leaf has label $1$.
			Let $y$ be the parent of $z$.
			Then $y$ is not the root and has at least two children.
			
			Recover $L$ and $H$ as follows:
			\begin{enumerate}
				\item
				Remove the subtree rooted at $z$ and attach its root
				to a new root with the same label as $z$.
				Denote this tree by $L$.
				
				\item
				In the remaining tree, decrease the label of each node
				on the path from the root to $y$ by $\ell-1$.
				Denote the resulting tree by $H$.
			\end{enumerate}
			
			The labeling conditions for $L$ and $H$ follow
			from these operations, and $L$ is indecomposable.
			The root of $H$ and its unique child both have label $1$.
			Remove the root of $H$ and replace the new root label
			by the sum of its children's labels to obtain $R$.
			Thus $H=\lambda_1(R)$, and both $L$ and $R$ are
			$\beta(1,0)$-trees with fewer edges than $N$.
			By the induction hypothesis and Proposition~\ref{prop:psi},
			we may recover $\alpha$ and $\beta$ from $L$ and $R$ by letting $L=\theta(\alpha)$ and $R=\psi(\beta)$.
			
			Using the notation $r_s$ and $b_s$ for $H$ as in the
			construction of $\theta(\pi)$, one may verify that
			$y=r_p$ and $z$ is the root of the inserted subtree.
			Set $q=\ell-\rootlab(L)$. Then
			\[
			i=\sum_{s=p}^{d}b_s-q+1,
			\]
			completing the proof that $\theta$ is injective.
			
			It remains to show that the map $\theta$ is surjective. In the following, we aim to show that
				$|\overline{\TT}_n|=|\overline{\B}_n|$ for all $n\geq 1$. We proceed by induction on $n$. Clearly, the assertion holds for $n=1$. Now we assume that $n\geq2$ and 	$|\overline{\TT}_m|=|\overline{\B}_m|$ holds  for all $m<n$.  
			By the unique direct-sum decompositions in
			Subsections~\ref{subsec:2ssp-decomposition} and ~\ref{subsec:tree-decomposition}, each decomposable
			object is uniquely determined by its first indecomposable
			component and the remaining nonempty object.
			Counting according to the size of the first component gives
			\[
			\begin{aligned}
				|\overline{\TT}_n|
				&=|\TT_n|-\sum_{k=1}^{n-1}
				|\overline{\TT}_k|\,|\TT_{n-k}|,\\
				|\overline{\B}_n|
				&=|\B_n|-\sum_{k=1}^{n-1}
				|\overline{\B}_k|\,|\B_{n-k}|.
			\end{aligned}
			\]
			Recall that we have
			$|\TT_m|=|\B_m|$ for every $m\geq0$.
			By the induction hypothesis, we have
		  $|\overline{\TT}_k|=|\overline{\B}_k|$
			for all $k<n$. Therefore, the two right-hand sides above are equal, which implies that
			$|\overline{\TT}_n|=|\overline{\B}_n|$
			for every $n\geq1$.
			Together with the injectivity of $\theta$,
			this proves the surjectivity and completes the proof.
		\end{proof}

	\subsection{The bijection $F$}
	\label{sec:bijection}
	 
	Set $F(\epsilon)=\circ$.
	For a nonempty permutation $\pi\in\TT$, write its
	unique decomposition into indecomposable permutations as
	\[
	\pi=\pi^{(1)}\oplus\cdots\oplus\pi^{(r)},
	\qquad r=\comp(\pi).
	\]
	Define
	\begin{equation}\label{eq:F-sum}
		F(\pi)
		=\theta(\pi^{(1)})\oplus\cdots\oplus\theta(\pi^{(r)}).
	\end{equation}
	In particular, $F(\pi)=\theta(\pi)$ when $\pi$ is
	indecomposable and $F(1)=\lambda_1(\circ)$.
	
	\medskip
	For example, let $\beta=13827564$ and
	$\pi=142(12)36(11)5(10)897$.
	We have $\pi=1\oplus C(21,\beta,1)$.
	$\theta(C(21,\beta,1))$ is the image
	for $i=1$ in Fig.~\ref{fig:Omega-example}.
	By~\eqref{eq:F-sum}, $F(\pi)$
	is shown in Fig.~\ref{fig:F-example}.
	In this example, both quadruples in~\eqref{eq:main-statistics} are equal to $(2,7,1,5)$.
	
	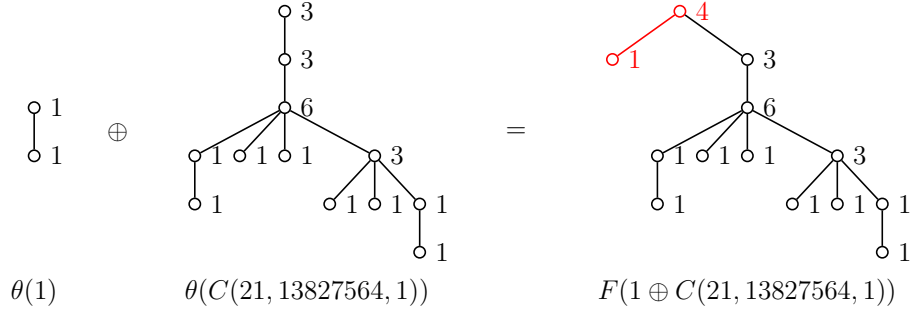
\begin{figure}[htbp]
		\centering
		\begin{tikzpicture}[
			scale=0.85,
			transform shape,
			x=1cm,y=1cm,
			vertex/.style={
				circle,draw,fill=white,
				inner sep=0pt,minimum size=5pt
			},
			every label/.style={font=\normalsize},
			tree caption/.style={
				align=center,anchor=north,inner sep=0pt
			},
			line width=0.6pt
			]
			\begin{scope}
				\node[vertex,label=right:$1$]
				(a0) at (0.00,-1.50) {};
				\node[vertex,label=right:$1$]
				(a1) at (0.00,-2.25) {};
				\draw (a0)--(a1);
				\node[tree caption] at (0.00,-4.15)
				{$\theta(1)$};
			\end{scope}
			
			\node at (1.30,-1.875) {$\oplus$};
			
			\begin{scope}
				\node[vertex,label=right:$3$]
				(b0) at (3.90,0.00) {};
				\node[vertex,label=right:$3$]
				(b1) at (3.90,-0.75) {};
				\node[vertex,label=right:$6$]
				(b2) at (3.90,-1.50) {};
				\node[vertex,label=right:$1$]
				(b3) at (2.50,-2.25) {};
				\node[vertex,label=right:$1$]
				(b4) at (2.50,-3.00) {};
				\node[vertex,label=right:$1$]
				(b5) at (3.20,-2.25) {};
				\node[vertex,label=right:$1$]
				(b6) at (3.90,-2.25) {};
				\node[vertex,label=right:$3$]
				(b7) at (5.30,-2.25) {};
				\node[vertex,label=right:$1$]
				(b8) at (4.60,-3.00) {};
				\node[vertex,label=right:$1$]
				(b9) at (5.30,-3.00) {};
				\node[vertex,label=right:$1$]
				(b10) at (6.00,-3.00) {};
				\node[vertex,label=right:$1$]
				(b11) at (6.00,-3.75) {};
				\draw
				(b0)--(b1) (b1)--(b2) (b2)--(b3)
				(b3)--(b4) (b2)--(b5) (b2)--(b6)
				(b2)--(b7) (b7)--(b8) (b7)--(b9)
				(b7)--(b10) (b10)--(b11);
				\node[tree caption] at (4.25,-4.15)
				{$\theta(C(21,13827564,1))$};
			\end{scope}
			
			\node at (7.50,-1.875) {$=$};
			
			\begin{scope}
				\node[vertex,draw=red,label={[red]right:$4$}]
				(c0) at (10.05,0.00) {};
				\node[vertex,draw=red,label={[red]right:$1$}]
				(c1) at (9.00,-0.75) {};
				\node[vertex,label=right:$3$]
				(c2) at (11.10,-0.75) {};
				\node[vertex,label=right:$6$]
				(c3) at (11.10,-1.50) {};
				\node[vertex,label=right:$1$]
				(c4) at (9.70,-2.25) {};
				\node[vertex,label=right:$1$]
				(c5) at (9.70,-3.00) {};
				\node[vertex,label=right:$1$]
				(c6) at (10.40,-2.25) {};
				\node[vertex,label=right:$1$]
				(c7) at (11.10,-2.25) {};
				\node[vertex,label=right:$3$]
				(c8) at (12.50,-2.25) {};
				\node[vertex,label=right:$1$]
				(c9) at (11.80,-3.00) {};
				\node[vertex,label=right:$1$]
				(c10) at (12.50,-3.00) {};
				\node[vertex,label=right:$1$]
				(c11) at (13.20,-3.00) {};
				\node[vertex,label=right:$1$]
				(c12) at (13.20,-3.75) {};
				\draw
				 (c0)--(c2) (c2)--(c3)
				(c3)--(c4) (c4)--(c5) (c3)--(c6)
				(c3)--(c7) (c3)--(c8) (c8)--(c9)
				(c8)--(c10) (c8)--(c11) (c11)--(c12);
				\draw[red]
				(c0)--(c1);
				\node[tree caption] at (11.10,-4.15)
				{$F(1\oplus C(21,13827564,1))$};
			\end{scope}
		\end{tikzpicture}
		\caption{An example of the bijection $F$.}
		\label{fig:F-example}
	\end{figure}
	
	Conversely, given a $\beta(1,0)$-tree $B\in \B_n$ with $n\geq 1$,  write its unique decomposition into indecomposable  $\beta(1,0)$-trees   as 
	$$
	B=B_1\oplus  \cdots \oplus B_r.
	$$
	Define
	$$
	F'(B)=\theta^{-1}(B_1)\oplus  \cdots \oplus \theta^{-1}(B_r).
	$$
	\begin{proof}[\normalfont\bfseries Proof of Theorem~\ref{thm:main}]
	By 
		Proposition~\ref{prop:theta}, the maps $F$ and $F'$ are inverses of each other, and hence  $F:\TT_n\longrightarrow\B_n$ is a bijection
		for every $n\geq1$.
		
		Let $\pi\in\TT_n$ have the indecomposable decomposition
		$\pi=\pi^{(1)}\oplus\cdots\oplus\pi^{(r)}$,
		and set $B_j=\theta(\pi^{(j)})$ for $1\leq j\leq r$.
		From the definition of the direct sum, we have
		\begin{equation}\label{eq:F-sum-permutation-statistics}
			\begin{aligned}
				\comp(\pi)&=r,\\
				\asc(\pi)&=r-1+\sum_{j=1}^{r}\asc(\pi^{(j)}),\\
				\ldr(\pi)&=\ldr(\pi^{(1)}),\\
				\rmax(\pi)&=\rmax(\pi^{(r)}).
			\end{aligned}
		\end{equation}
		Since each $B_j$ is indecomposable, we obtain
		\[
		\begin{array}{lll}
			\sub(F(\pi))
			&=r
			&\mbox{(by \eqref{eq:F-sum},
				\eqref{eq:tree-sum-statistics})}\\
			&=\comp(\pi)
			&\mbox{(by \eqref{eq:F-sum-permutation-statistics})},
			\\[5pt]
			\leaves(F(\pi))
			&=\displaystyle\sum_{j=1}^{r}\leaves(B_j)
			&\mbox{(by \eqref{eq:F-sum},
				\eqref{eq:tree-sum-statistics})}\\
			&=r+\displaystyle\sum_{j=1}^{r}\asc(\pi^{(j)})
			&\mbox{(by \eqref{eq:Omega-statistics})}\\
			&=1+\asc(\pi)
			&\mbox{(by \eqref{eq:F-sum-permutation-statistics})},
			\\[5pt]
			\lsub(F(\pi))
			&=\lsub(B_1)
			&\mbox{(by \eqref{eq:F-sum},
				\eqref{eq:tree-sum-statistics})}\\
			&=\ldr(\pi^{(1)})
			&\mbox{(by \eqref{eq:Omega-statistics})}\\
			&=\ldr(\pi)
			&\mbox{(by \eqref{eq:F-sum-permutation-statistics})},
			\\[5pt]
			\rpath(F(\pi))
			&=\rpath(B_r)
			&\mbox{(by \eqref{eq:F-sum},
				\eqref{eq:tree-sum-statistics})}\\
			&=\rmax(\pi^{(r)})
			&\mbox{(by \eqref{eq:Omega-statistics})}\\
			&=\rmax(\pi)
			&\mbox{(by \eqref{eq:F-sum-permutation-statistics})}.
		\end{array}
		\]
		This proves~\eqref{eq:main-statistics}, completing the proof.
	\end{proof}

	 \section*{Acknowledgments}
	 This work was supported by the National  Science Foundation of China  (grants 12471318).

	%%%%%%%%%%%%%%%%%%%%%%%%%%%

\end{document}